\documentclass[12pt,reqno]{amsart}
\usepackage[top=1in, bottom=1in, left=1in, right=1in]{geometry}
\usepackage{times, amsthm, amssymb, amsmath, amsfonts, amsthm, bm, graphicx, mathrsfs}
\usepackage[all]{xy}
\usepackage[usenames,dvipsnames]{color}
\usepackage[pagebackref=true,pdftex]{hyperref}

\newtheorem{theorem}{Theorem}[section]
\newtheorem{lemma}[theorem]{Lemma}

\newtheorem{cor}[theorem]{Corollary}
\newtheorem{prop}[theorem]{Proposition}

\theoremstyle{definition}

\newtheorem{remark}[theorem]{Remark}

\newtheorem{definition}[theorem]{Definition}

        {\begin{list}
                {\noindent\makebox[0mm][r]{\arabic{enumi}.}}
                {\leftmargin=5.5ex \usecounter{enumi}}
        }
        {\end{list}}

        {\begin{list}
                {\noindent\makthus,ebox[0mm][r]{(\roman{enumi})}}
                {\leftmargin=5.5ex \usecounter{enumi}}
        }
        {\end{list}}

\def\<{\langle}
\def\>{\rangle}
\def\0{\mathbf{0}}

\def\CC{{\mathbb C}}

\def\EE{{\mathbb E}}
\def\cE{{\mathcal E}}

\def\NN{{\mathbb N}}

\def\QQ{{\mathbb Q}}

\def\RR{{\mathbb R}}

\def\ZZ{{\mathbb Z}}

\def\bbE{{\mathbb E}}

\def\cJ{{\mathcal J}}

\def\del{\partial}

\def\vol{{\rm \operatorname{vol}}}

\def\rank{{\operatorname{rank}}}

\def\codim{{\rm codim}}
\def\minus{\smallsetminus}
\def\nothing{\varnothing}

\newcommand*{\defeq}{\mathrel{\vcenter{\baselineskip0.5ex \lineskiplimit0pt
                     \hbox{\scriptsize.}\hbox{\scriptsize.}}}%
                     =}

\numberwithin{equation}{section}
\begin{document}%%%%%%%%%%%%%%%%%%%%%%%%%%%%%
%%%%%%%%%%%%%%%%%%%%%%%%%%%%%%%%%%%%%%

\mbox{}
\title[$A$-hypergeometric rank versus the volume of $A$]{On the rank of an $A$-hypergeometric $D$-module \\ versus the normalized volume of $A$} 
 
\author{Christine Berkesch}
\address{School of Mathematics \\
University of Minnesota.}
\email{cberkesc@umn.edu}

\author{Mar\'ia-Cruz Fern\'andez-Fern\'andez}
\address{Departamento de \'Algebra \\
Universidad de Sevilla.}
\email{mcferfer@algebra.us.es}

\thanks{CB was partially supported by NSF Grant DMS 1661962. MCFF was partially supported by MTM2016-75024-P and FEDER}


\begin{abstract}
The rank of an $A$-hypergeometric $D$-module $M_A(\beta)$, associated with a full rank $(d\times n)$-matrix $A$ and a vector of parameters $\beta\in \CC^d$, is known to be the normalized volume of $A$, denoted $\vol(A)$, when $\beta$ lies outside the exceptional arrangement $\cE(A)$, an affine subspace arrangement of codimension at least two. If $\beta\in \cE(A)$ is simple, we prove that $d-1$ is a tight upper bound for the ratio $\rank(M_A(\beta))/\vol(A)$ for any $d\geq 3$.  We also prove that the set of parameters $\beta$ such that this ratio is at least $2$ is an affine subspace arrangement of codimension at least $3$.
\end{abstract}
\maketitle

\mbox{}
\vspace{-14mm}
\parskip=0ex
\parindent2em
\parskip=1ex
\parindent0pt

%%%%%%%%%%%%%%%%%%%%%%%%%%%%%%%%%%%%%%
\setcounter{section}{0}
\section*{Introduction}
%%%%%%%%%%%%%%%%%%%%%%%%%%%%%%%%%%%%%%
\vspace{-2mm}

The systematic study of $A$-hypergeometric $D$-modules, also known as GKZ-systems, was initiated by Gelfand, Graev, Kapranov, and Zelevinski \cite{GGZ}, \cite{GKZ}. These are systems of linear partial differential equations in several complex variables that generalize classical hypergeometric equations.
They are determined by a matrix $A=(a_1 \cdots a_n)=(a_{i,j})$ with columns $a_k\in\ZZ^d$ and a parameter vector $\beta\in\CC^d$. 
Let $x_1,x_2,\dots,x_n$ be coordinates on $\CC^n$, with corresponding partial derivatives $\del_1,\del_2,\dots,\del_n$, so that the Weyl algebra $D$ on $\CC^n$ is generated by $x_1,\dots,x_n,\del_1,\dots,\del_n$. 
Let 
\[
I_A \defeq \<\del^u-\del^v\mid  u,v\in\NN^n, Au=Av\> \subseteq \CC[\del_1,\ldots, \del_n]
\] 
denote the \emph{toric ideal} of $A$. 
Denote by $E_i\defeq\sum_{j=1}^n a_{i,j} x_j \del_j$ the $i$th \emph{Euler operator} of $A$. The \emph{$A$-hypergeometric $D$-module} with parameter $\beta\in\CC^d$ is the left $D$-module 
\[
M_A(\beta)\defeq D/D\cdot\< I_A, E_1-\beta_1,\ldots, E_d-\beta_d\>.
\]

For any choice of $A$ and $\beta$, the module $M_A(\beta)$ is \emph{holonomic} \cite{GGZ,adolphson}. When $\beta\in\CC^d$ is generic, the dimension of the space of germs of holomorphic solutions of $M_A(\beta)$ at a nonsingular point, also known as its \emph{(holonomic) rank}, is equal to the \emph{normalized volume} $\vol(A)$ of the matrix $A$, see \eqref{eqn:normalized-volume} \cite{GKZ, adolphson}. In general, this is only a lower bound; see \cite{SST} for the case when $I_A$ is homogeneous and \cite{MMW} for the general case. 
The set 
\[
\cE(A)\defeq\{\beta\in\CC^d \mid \; \rank (M_A (\beta))>\vol(A)\} 
\]
is called the \emph{exceptional arrangement} of $A$, which is an affine subspace arrangement of codimension at least two that is closely related to the local cohomology modules of the toric ring $\CC[\del]/I_A$ \cite{MMW}. 
A parameter $\beta \in \cE(A)$ is called a \emph{rank jumping} parameter.

There is a combinatorial formula to compute the rank of $M_A(\beta)$ in terms of the \emph{ranking lattices} $\EE^\beta$ of $A$ at $\beta$ \cite{berkesch}, with previous results in 
\cite{CDD} with $d=2$ and in 
\cite{okuyama} when $d=3$ or $\beta$ is \emph{simple} (see also Section \ref{sec:Simple-rank-jumping}). 
Unfortunately, the presence of alternating signs in this formula do not yield a strong upper bound for the rank of $M_A(\beta)$; however, if $\beta$ is simple, it quickly follows that the rank of $M_A(\beta)$ is at most $(d-1)\vol(A)$, see Corollary \ref{cor:upper-bound-F-simple-rank}. 
We show that this bound is tight by constructing a sequence of examples for which the ratio $\rank(M_A(\beta))/\vol(A)$ tends to $d-1$, see Theorem \ref{thm:family-examples}. 
In addition, we prove that the equality cannot hold for any example with simple parameter $\beta$ and that our examples are minimal in certain sense, see Remark \ref{rem:remark-minimal-volume}. Another interesting feature of these examples is that $\cE(A)$ contains all the lattice points in the convex hull of the columns of $A$ and the origin.

On the other hand, there are other known upper bounds for the holonomic rank of $M_A(\beta)$. In particular, 
\[
\rank (M_A (\beta))\leq \begin{cases}
4^d \cdot \vol(A) & \text{if $I_A$ is homogeneous~\cite{SST},}\\
4^{d+1}\cdot \vol(A) & \text{otherwise~\cite{BFM-parametric}.}
\end{cases}
\]
It was shown in \cite{Fer-exp-growth} that these upper bounds are qualitatively effective, i.e., there is some $a>1$ such that for any $d\geq 3$, there is a $(d\times n$)-matrix $A_d$ and a parameter $\beta_d\in \CC^d$ such that 
\[
\rank(M_{A_{d}}(\beta_d))\geq a^d \vol(A_d).
\] 
However, the maximum possible value of $\sqrt[d]{\rank(M_A(\beta))/\vol(A)}$ that has, up until now, appeared in the literature is $\sqrt[3]{7/5}\approx 1.1187$, see \cite[Example 2.6]{Fer-exp-growth}, which was first considered in \cite{MW}. The supremum of the value of $\sqrt[d]{\rank(M_A(\beta))/\vol(A)}$ over the examples in the current note is $\sqrt[5]{4}\approx 1.3195$,  i.e., $\sqrt[d]{d-1}$ for $d=5$. It is still an open problem to find the supremum of the set of values of $\sqrt[d]{\rank(M_A(\beta))/\vol(A)}$ for variation among the set of full rank $(d\times n)$-matrices $A$ and $\beta\in \CC^d$, for $d\geq 3$ and $n\geq d+2$. 

%%%%%%%%%%%%%%%%%%%%%%%%%%%%%%%%%%%%%
% \subsection*{Outline}
%%%%%%%%%%%%%%%%%%%%%%%%%%%%%%%%%%%%%

%%%%%%%%%%%%%%%%%%%%%%%%%%%%%%%%%%%
\subsection*{Acknowledgements}
%%%%%%%%%%%%%%%%%%%%%%%%%%%%%%%%%%%
We are grateful to Laura Felicia Matusevich and Uli Walther for helpful discussions over the years on bounding the rank of an $A$-hypergeometric system. 

%%%%%%%%%%%%%%%%%%%%%%%%%%%%%%%%%%%
\section{Lower bounds for the normalized volume}
%%%%%%%%%%%%%%%%%%%%%%%%%%%%%%%%%%%
\label{sec:lowerBounds}

In this section, we recall the definition of normalized volume of an integer full rank matrix, see \eqref{eqn:normalized-volume}, and provide some lower bounds for it, see Lemma \ref{lem:volume-lower-bound} and Corollary \ref{cor:nonCM-inequality}. These bounds will be used in the proof of Corollary \ref{cor:sharper-upper-bound-F-simple-rank}.

Fix a $(d\times n)$-integer matrix $A=(a_1 \cdots a_n)$, where $a_i\in\ZZ^d$ denotes the $i$th column of $A$. 
With the convention that $0\in\NN$, 
assume that $\ZZ A\defeq \sum_{j=1}^n \ZZ a_j =\ZZ^d$ and that the affine semigroup $\NN A\defeq\sum_{j=1}^n \NN a_j$ is positive, meaning that $\NN A\cap(-\NN A)=\{\bf{0}\}$. 
We also assume for simplicity that all the columns of $A$ are distinct from each other and the origin. 

Identify $A$ with its set of columns, and for any subset $F$ of $A$, denote by $\Delta_{F}$ the convex hull in $\RR^d$ of the origin and $F$. 
We also identity $F$ with its index set $\{j\mid  a_j \in F\}$. 
Given a lattice $\Lambda$ such that $F\subseteq \Lambda\subseteq \QQ F\cap\ZZ^d$, 
the \emph{normalized volume} of $F$ in $\Lambda$ is the integer 
\begin{equation}\label{eqn:normalized-volume}
\vol_{\Lambda}(F)
= \dim(\RR F)! \cdot 
\dfrac{\vol_{\RR F}(\Delta_F)}{[ \ZZ^d\cap \QQ F : \Lambda]}, 
\end{equation}
where $\vol_{\RR F}(\cdot)$ denotes Euclidean volume in $\RR F$. We write $\vol(A)$ for $\vol_{\ZZ A}(A) = \vol_{\ZZ^d}(A)$. By convention, $\ZZ \nothing =\{\bf{0}\}$ and $\vol_{\{\bf{0}\}}(\nothing)=1$.

A subset $F$ of the columns of the matrix $A$ is a \emph{face} of $A$, denoted $F\preceq A$, if $\RR_{\geq0}F$ is a face of the cone $\RR_{\geq0}A\defeq\sum_{j=1}^n\RR_{\geq0} a_j$ and $F= A\cap\RR F$. 
The codimension of a nonempty face $F$ of $A$ is $\codim(F)\defeq d-\dim (\RR F)$, with the convention that $\codim (\nothing)=d$. 

\begin{lemma}\label{lem:adding-a-point}
If $\tau$ is a proper subset of $A$ with $\Delta_\tau \cap A=\tau$, then there exists a column $a$ of $A\setminus \tau$ such that $\Delta_{\tau\cup\{a\}}\cap A=\tau\cup\{a\}$. Moreover, if $\Delta_{\tau}$ is not full dimensional, then $a$ may be chosen so that $\dim (\Delta_{\tau\cup\{a\}})=\dim (\Delta_\tau)+1$.
\end{lemma}
\begin{proof} 
If $\Delta_\tau\subseteq \RR^d$ is full dimemsional, then choose any column $a\in A\setminus \tau$. 
Since $a\notin \Delta_\tau$, the vector $a$ is a vertex of $\Delta_{\tau \cup\{a\}}$, and the rest of the vertices of $\Delta_{\tau \cup\{a\}}$ are vertices of $\Delta_\tau$. 
In particular, if there exists a vector $a'\in (\Delta_{\tau\cup\{a\}}\cap A)\setminus(\tau\cup\{a\})$, then $a'$ is not a vertex of $\Delta_{\tau\cup\{a\}}$ and $\Delta_\tau\subsetneq\Delta_{\tau\cup\{a'\}} \subsetneq \Delta_{\tau\cup\{a\}}$. 
Thus, $a$ can be replaced by $a'$. Also, notice that 
\[
(\Delta_{\tau\cup\{a'\}}\cap A)\setminus(\tau\cup\{a'\})\subsetneq(\Delta_{\tau\cup\{a\}}\cap A)\setminus(\tau\cup\{a\}).
\] 
We can thus repeat this process of replacement of $a$ until the equality $\Delta_{\tau\cup\{a''\}}\cap A=\tau\cup\{a''\}$ holds for some $a''$ in $A\setminus\tau$. 

On the other hand, if $\Delta_\tau\subseteq \RR^d$ is not full dimensional, 
let $a\in A\setminus \tau$ be such that $\dim (\Delta_{\tau\cup\{a\}})=\dim (\Delta_\tau)+1$. Such a choice of $a$ exists because the rank of $A$ is $d$. 
Since $\Delta_\tau$ is a facet of $\Delta_{\tau \cup\{a\}}$ and $\Delta_\tau \cap A=\tau$, no point in $(\Delta_{\tau \cup\{a\}}\cap A)\setminus \tau$ is in the affine span of $\tau$. 
Thus, if there exists a vector $a'\in (\Delta_{\tau\cup\{a\}}\cap A)\setminus(\tau\cup\{a\})$, we can replace $a$ by $a'$ and repeat the process until $(\Delta_{\tau\cup\{a\}}\cap A)=\tau\cup\{a\}$, in a similar way as in the full dimensional case.
\end{proof}

\begin{lemma}\label{lem:volume-lower-bound}
If $F\preceq A$ is a face of $A$, then 
\begin{equation}\label{eqn:volume-F-ineq}
\vol(A)\geq \vol_{\ZZ^d\cap \QQ F}(F)+n-|F| -\codim (F),
\end{equation} where $|F|$ is the cardinality of $F$. In particular, $\vol(A)\geq n-d +1$.
\end{lemma}
\begin{proof}
Since $F\preceq A$ is a face of $A$, $\Delta_F\cap A=F$. 
By Lemma~\ref{lem:adding-a-point}, there is a set $\sigma$ of $\codim(F)$ linearly independent columns of $A\setminus F$ such that $\Delta_{F\cup\sigma}$ is full dimensional and $\Delta_{F\cup\sigma}\cap A=F\cup\sigma$. 
The normalized volume in the lattice $\ZZ^d$ of $F\cup\sigma$ is at least $\vol_{\ZZ^d\cap \QQ F}(F)$. To see this, denote $\widetilde{F}=\ZZ^d \cap \Delta_F$ and notice that $\RR F \cap \ZZ^d=\ZZ \widetilde{F}$ and $\Delta_{\widetilde{F}}=\Delta_F$. Thus, using \eqref{eqn:normalized-volume} we obtain
\[
\vol_{\ZZ^d}(F\cup \sigma )=\vol_{\ZZ^d}(\widetilde{F}\cup \sigma )\geq \vol_{\ZZ (\widetilde{F}\cup \sigma)}(\widetilde{F}\cup \sigma )\geq \vol_{\ZZ \widetilde{F}}(\widetilde{F})=\vol_{\ZZ^d \cap \RR F}(F),
\]
where the first inequality follows from the contaiment $\ZZ (\widetilde{F}\cup \sigma)\subseteq \ZZ^d$ and the second one follows from \cite[Lemma 3.13]{duality}.

Again by Lemma~\ref{lem:adding-a-point}, there is a column $a$ of $A\setminus (F\cup \sigma)$ such that no other column of $A\setminus (F\cup \sigma)$ lies in 
$\Delta_{F\cup\sigma\cup\{a\}}$, the convex hull of the $\codim(F)+|F|+1$ points of $F\cup \sigma\cup\{a\}$ and the origin. 
In fact, $n-(\codim(F)+|F|+1)$ more columns of $A\setminus(F\cup\sigma)$ can be iteratively found in this way. Notice that each time a new point is added to $F\cup\sigma$ using Lemma~\ref{lem:adding-a-point}, the normalized volume of the convex hull of the new set is increased at least by one. This proves the first statement. The second statement follows from the first one by taking $F=\nothing$.
\end{proof}

\begin{remark}\label{rem:pyramid}
Notice that for any face $F\preceq A$, $n-|F| -\codim (F)\geq 0$. If equality holds, we say that $A$ is a \emph{pyramid} over $F$. By \cite[Lemma 3.5]{reducibility}, $A$ is a pyramid over $F$ if and only if $\ZZ^d= \ZZ F \oplus \left(\bigoplus_{j\notin F} \ZZ a_j\right)$. Further, if $A$ is a pyramid over $F$, then equality holds in \eqref{eqn:volume-F-ineq} because $\ZZ^d\cap \QQ F=\ZZ F$ and $\vol(A)=\vol_{\ZZ F}(F)$, see \cite[Lemma 3.5]{reducibility}. 
The converse is not true; a counterexample is provided in Remark \ref{rem:remark-minimal-volume}.
\end{remark}

% We sat that $A$ is a \emph{pyramid} over $F$ if $\ZZ^d= \ZZ F \oplus \left(\bigoplus_{j\notin F} \ZZ a_j\right)$. Notice that for any face $F\preceq A$, $n-|F| -\codim (F)\geq 0$, and equality holds if and only if $A$ is a pyramid over $F$. Further, if $A$ is a pyramid over $F$, then equality holds in \eqref{eqn:volume-F-ineq} because $\ZZ^d\cap \QQ F=\ZZ F$ and $\vol(A)=\vol_{\ZZ F}(F)$, see \cite[Lemma 3.5]{reducibility}. 
% The converse is not true; a counterexample is provided in Remark \ref{rem:remark-minimal-volume}.

On the other hand, if equality holds in \eqref{eqn:volume-F-ineq}, then all the lattice points in $\Delta_A\setminus \Delta_F$ are columns of $A$. Indeed, if there is a lattice point $a \in \Delta_A\setminus \Delta_F$ which is not a column of $A$, then a matrix $A'$ obtained by adding to $A$ the column $a$ would have $n+1$ columns and $\vol(A)=\vol(A')$, so inequality \eqref{eqn:volume-F-ineq} applied to $A'$ shows that the inequality corresponding to $A$ cannot be an equality in this case.

Denote the toric ring associated to $A$ by $S_A\defeq\CC[\del]/I_A\cong \CC [\NN A]$. 

\begin{prop}\label{prop:volume-normal}
If $\vol(A)= n-d+1$, then $S_A$ is normal.
\end{prop}
\begin{proof}
Let $H$ be an affine hyperplane such that $\tau\defeq H\cap A$ satisfies that $\Delta_\tau$ is full dimensional and all the columns of $A$ not in $\tau$ belong to the open half space determined by $H$ not containing the origin; such a hyperplane exists because $\NN A$ is positive and $\bf{0}$ is not a column of $A$. 
For any $\sigma\subseteq \tau$ of cardinality $d$ such that $\Delta_\sigma$ is a full dimemsional simplex and $\Delta_{\sigma}\cap A=\sigma$, $\vol_{\ZZ^d}(\sigma)=1$ because otherwise, by adding a point of $A\setminus\sigma$ using Lemma~\ref{lem:adding-a-point} and taking the convex hull iteratively, the volume would increase by at least one in each step and the normalized volume of $A$ would be larger than $n-d+1$. 
Now, since $\vol_{\ZZ^d}(\sigma)=1$, $\sigma$ forms a basis in the lattice $\ZZ^d$ and $\NN \sigma =\ZZ^d \cap \RR_{\geq 0} \sigma$. Since $\RR_{\geq 0}A$ equals the union of the cones $\RR_{\geq 0}\sigma$ for simplices $\sigma\subseteq \tau$ with $\tau$ as above, it follows that $\RR_{\geq 0}A\cap \ZZ^d =\NN A$, and hence, $S_A$ is normal.
\end{proof}

\begin{cor}\label{cor:nonCM-inequality}
If $S_A$ is not Cohen--Macaulay, then $d\geq 2$, $n\geq d+2$, and $\vol(A)\geq n-d+2$.
\end{cor}
\begin{proof}
If either $d=1$ or $n-d=1$, then under our hypotheses on $A$, $S_A$ is Cohen--Macaulay. On the other hand, if $\vol(A)< n-d+2$, then $S_A$ is normal by Lemma~\ref{lem:volume-lower-bound} and Proposition~\ref{prop:volume-normal}, which implies that $S_A$ is Cohen--Macaulay by~\cite[Theorem 1]{Hochster}.
\end{proof}

The inequality in Corollary~\ref{cor:nonCM-inequality} is sharp; for any $d\geq 2$ and $n\geq d+2$, there is a pointed matrix $A$ as above with $\vol(A)=n-d+2$ such that $S_A$ is not Cohen--Macaulay. 
To see this, notice first that for $d=2$ and $n=d+2=4$, the matrix 
\[
A=\left(\begin{array}{cccc}
1 & 1 & 0 & 0\\
0 & 1 & 2 & 3\end{array}\right)
\] 
satisfies that $\vol(A)=4$ and $S_A$ is not Cohen--Macaulay. On the other hand, in order to produce examples with $n\geq 5$, it is enough to modify this example by adding the columns $(0,k)^t$ for $k=4, \ldots, n-1$, and this operation keeps $S_A$ invariant up to isomorphism.
To construct more examples with the same value of $n-d$ but larger $d$, it is enough to consider a pyramid over the previous example. This alters $S_A$ by tensoring over $\CC$ with a polynomial ring in a number of variables equal to the increment of $d$.

%%%%%%%%%%%%%%%%%%%%%%%%%%%%%%%%%%%%%%
\section{Rank versus volume in the simple case}
%%%%%%%%%%%%%%%%%%%%%%%%%%%%%%%%%%%%%%
\label{sec:Simple-rank-jumping}

\subsection{Combinatorics of the rank}\label{subsection1}
In this subsection, we recall some notations and results from \cite{berkesch} and a formula for the rank of an $A$-hypergeometric system in a particular case, see \eqref{eqn:formula-simple-rank-jump}, proved in \cite{okuyama}. 

For a face $F\preceq A$, consider the union of the lattice translates 
\begin{align*}
\EE_F^\beta\defeq 
\big[\ZZ^d\cap(\beta+\CC F) \big]\minus(\NN A+\ZZ F) = \bigsqcup_{b\in B_F^\beta} (b+\ZZ F),
\end{align*}
where $B_F^\beta \subseteq \ZZ^d$ is a set of lattice translate representatives. 
As such, $|B^\beta_{F}|$ is the number of translates of $\ZZ F$ appearing in $\bbE_F^\beta$, which is by definition equal to the difference between $[\ZZ^d\cap \QQ F:\ZZ F]$ and the number of translates of $\ZZ F$ along $\beta+\CC F$ that are contained in $\NN A + \ZZ F$.

Given the set
$\mathcal{J}(\beta)\defeq\{(F,b)\mid F\preceq A,\, b\in B_F^\beta \}$, 
the \emph{ranking lattices} of $A$ at $\beta$ are defined to be 
\begin{align*}
\EE^\beta 
  \defeq \bigcup_{(F,b)\in \cJ(\beta)} (b+\ZZ F).
\end{align*}
Note that the ranking lattices of $A$ at $\beta$ is precisely the union of those sets $(b+\ZZ F)$ contained in $\ZZ^d \setminus \NN A$ such that $\beta\in (b+\CC F)$. This is closely related to the set of holes of the affine semigroup $\NN A$, namely the set $(\ZZ^d \cap \RR_{\geq 0}A)\setminus \NN A$. 

\begin{definition}
\label{def:simple}
A rank jumping parameter $\beta$ is \emph{simple} (for a face $G\preceq A$) if the set of maximal pairs $(F,b)$ in $\cJ(\beta)$ with respect to inclusion on $b+\ZZ F$ all correspond to a unique face $G\preceq A$.
\end{definition}

The main result in \cite{berkesch} states that the rank of $M_A(\beta)$ can be computed from the combinatorics of $\EE^\beta$ and $\Delta_A$. An explicit formula for the rank is given when the rank jumping parameter $\beta$ is \emph{simple} for a face $G\preceq A$ (see also \cite{okuyama} for this particular case); in this case, 
\begin{equation}\label{eqn:formula-simple-rank-jump}
\rank (M_A (\beta))=\vol(A)+ |B_G^\beta|\cdot (\codim(G)-1)\cdot \vol_{\ZZ G}(G).
\end{equation}

\subsection{Some upper bounds for the rank}\label{subsection2}
In this subsection, we use \eqref{eqn:formula-simple-rank-jump} to provide upper bounds for the rank of an $A$-hypergeometric system $M_A(\beta)$ when the parameter $\beta$ is simple, see Corollaries \ref{cor:upper-bound-F-simple-rank} and \ref{cor:sharper-upper-bound-F-simple-rank}. We also prove that $\rank(M_A(\beta)) < 2 \cdot\vol (A)$ if $\beta$ lies outside an affine subspace arrangement of codimension at least three, see Theorem \ref{thm:codimension-three}. 

\begin{cor}\label{cor:upper-bound-F-simple-rank}
If $\beta\in\CC^d$ is simple for the face $F\preceq A$, then 
\[
\rank (M_A (\beta))\leq \codim(F)\cdot \vol(A). 
\]
In particular, if $d\geq 3$ and $\beta\in\cE(A)$ is simple, then 
\[
\rank (M_A (\beta))\leq (d-1)\cdot \vol(A).
\]
\end{cor}
\begin{proof}
The first statement follows from~\eqref{eqn:formula-simple-rank-jump} and the definition of normalized volume in \eqref{eqn:normalized-volume}. Indeed,
\begin{equation}\label{eqn:inequalities}
|B_F^\beta|\cdot \vol_{\ZZ F}(F)\leq [\ZZ^d \cap \QQ F: \ZZ F]\cdot \vol_{\ZZ F}(F)
=\vol_{\QQ F\cap \ZZ^d}(F)\leq \vol(A).
\end{equation}
We can assume without loss of generality that $\vol(A)\geq 2$, since otherwise $A$ is a simplex and $\cE(A)=\nothing$. 

For the second statement, notice first that if $\codim(F)=d$, then 
$\vol_{\ZZ F}(F)=1=|B_F^\beta|$ and 
\begin{equation}
\rank (M_A(\beta))=\vol(A)+d-1 \leq (d-1)\cdot \vol(A),\label{eqn:ineq2}
\end{equation} 
since $d\geq 3$ and $\vol(A)\geq 2$. 
Thus, we can assume that $\codim(F)\leq (d-1)$ and the second upper bound follows from the first one.
\end{proof}

We can improve the bound in Corollary~\ref{cor:upper-bound-F-simple-rank} as follows.

\begin{cor}\label{cor:sharper-upper-bound-F-simple-rank}
If $d\geq3$ and $\beta\in\CC^d$ is simple for the face $F\preceq A$, then 
\begin{equation}\label{eqn:sharper-inequality}
\rank (M_A (\beta))\leq \codim(F)\cdot \vol(A)-(\codim(F)-1)(n-|F|-\codim(F)). 
\end{equation}
In particular, if $\beta\in\cE(A)$ is simple, then 
\begin{equation}\label{eqn:strict-inequality-for-simple}
\frac{\rank (M_A (\beta))}{\vol(A)} < (d-1).
\end{equation}
\end{cor}
\begin{proof}
From \eqref{eqn:formula-simple-rank-jump} and the first inequality in \eqref{eqn:inequalities}, 
\begin{equation}
\label{eqn:eqn:sharper-inequality-first-step}
\rank (M_A (\beta))\leq \vol(A) + \vol_{\QQ F\cap \ZZ^d}(F) (\operatorname{codim}(F)-1). 
\end{equation}
Now $\vol_{\QQ F\cap \ZZ^d}(F)$ in \eqref{eqn:eqn:sharper-inequality-first-step} can be bounded using \eqref{eqn:volume-F-ineq} in order to obtain~\eqref{eqn:sharper-inequality}.

For~\eqref{eqn:strict-inequality-for-simple}, notice first that if $\codim(F)=1$, then $\rank (M_A(\beta))=\vol(A)$ by \eqref{eqn:formula-simple-rank-jump}; otherwise, $(\codim(F)-1)\geq 1$. 
By \cite[Corollary~9.2]{MMW}, $\cE(A)= \nothing$ is equivalent to $S_A$ being Cohen--Macaulay.
Thus, for the case when $\codim(F)=d$, it is enough to use that the inequality in \eqref{eqn:ineq2} is in fact strict, because $\vol(A)\geq 4$ by Corollary \ref{cor:nonCM-inequality}. 

For the remaining cases, the second part in the inequality~\eqref{eqn:sharper-inequality} is bounded above by 
\[
(d-1)\vol (A)-(n-|F|-\codim(F)).
\] 
Thus, it is enough to see that $n-|F|-\codim(F)\geq 1$. 
By way of contradiction, assume that $n-|F|-\codim(F)=0$ (i.e., $A$ is a pyramid over $F$), so that any $\beta\in\CC^d$ can be written uniquely as $\beta=\beta_F +\beta_{\overline{F}}$ with $\beta_F \in \CC F$, $\beta_{\overline{F}}\in \CC \overline{F}$ for $\overline{F}:=A \setminus F$ and 
\begin{equation}\label{eqn:A-F-equal-rank}
 \rank (M_A(\beta))= \rank( M_F(\beta_F)),
 \end{equation}
see \cite[Lemma 3.7]{reducibility}. It follows from \eqref{eqn:A-F-equal-rank} and Remark \ref{rem:pyramid} that $\beta\in\cE(A)$ if and only if $\beta_F\in\cE(F)$. Notice also that if $F\preceq G\preceq A$, then $\EE_G^\beta=\EE_G^{\beta'}$ for any $\beta'\in \beta+\CC F$. 
If $\beta$ is simple for $F$, the generic vectors $\beta' \in \beta + \CC F$ are also simple for $F$ and $\rank(M_A(\beta'))=\rank (M_A(\beta))$. 
Thus, 
\[
\rank (M_F(\beta'_F)) =\rank (M_F(\beta_F))>\vol_{\ZZ F}(F).
\] 
It follows that for generic $\gamma \in \CC F$, $\rank (M_F(\gamma))>\vol_{\ZZ F}(F)$, which is a contradiction, as this should be equality by~\cite{adolphson}.
\end{proof}

Notice that the difference between \eqref{eqn:strict-inequality-for-simple} and the second
statement of Corollary~\ref{cor:upper-bound-F-simple-rank} is that~\eqref{eqn:strict-inequality-for-simple} is a strict inequality.

\begin{theorem}\label{thm:codimension-three}
The set 
\[
\cE_2(A)\defeq\{\beta \in \CC^d \mid \; \rank (M_A (\beta))\geq 2 \cdot \vol(A)\}
\] 
is an affine subspace arrangement of codimension at least three in $\CC^d$.
\end{theorem}
\begin{proof}
The exceptional arrangement $\cE(A)$ is known to be a finite union of translates of linear subspaces $\CC G$ for faces $G\preceq A$ of codimension at least two \cite[Corollary 9.4 and Porism 9.5]{MMW}.  Moreover, it is shown in \cite[Theorem~2.6]{MMW} that rank of $M_A(\beta)$ is upper-semicontinuous as a function of $\beta$ with respect to the Zariski topology. Thus, on each irreducible component $C$ of $\cE(A)$ the rank of $M_A(\beta)$ is constant outside a Zariski closed subset of $C$, of codimension at least one in $C$ (i.e., of codimension at least three in $\CC^d$). Moreover, this codimension three set is also an affine subspace arrangement; see the argument after Definition 4.7 in \cite{berkesch}. It is thus enough to find, for any codimension two component $C$, a set of parameters $\beta\in C$ such that $\rank (M_A(\beta))<2 \cdot \vol(A)$ and whose Zariski closure is $C$. Indeed, if $C$ has codimension two, we have that $C=b+ \CC G$ for some face $G\preceq F$ of codimension two and some $b\in \CC^d$. 

Notice that for any proper face $G'\preceq A$ not containing $G$, the intersection $C\cap (\ZZ^d +\CC G')$ is at most a countably and locally finite union of translates of the linear space $\CC G\cap \CC G'$ of codimension at least three. 
Since there are only finitely many such faces $G'$, the Zariski closure of the set 
\[
C':=C\setminus \bigcup_{G\neq G'\preceq A}(\ZZ^d +\CC G') 
\]
is $C$. Moreover, for $G'$ as above and $\beta \in C'$, we have that $\ZZ^d \cap (\beta +\CC G')=\nothing$, hence $\EE^\beta_{G'}=\nothing$. In particular, the only possible faces involved in $\mathcal{J}(\beta)$ are $G$ and the two facets containing $G$. Thus, \cite[Section 5.3 and Example 6.21]{berkesch} yield the inequality
\begin{equation*}
\rank (M_A (\beta))\leq\vol(A)+ |B_G^\beta| (\codim(G)-1) \vol_{\ZZ G}(G),\end{equation*} 
where equality holds if $\beta$ is simple for $G$, as in Definition~\ref{def:simple}. From \eqref{eqn:sharper-inequality} applied to the codimension two face $G$, we obtain that 
\begin{equation}
\rank (M_A(\beta))\leq 2 \cdot  \vol(A) -(n-|G|-2)< 2 \cdot \vol(A)
\end{equation}
where $n-|G|-2\geq 1$ holds by the same argument as in the proof of Corollary~\ref{cor:sharper-upper-bound-F-simple-rank}. 
Thus, $\rank (M_A(\beta))<2\cdot  \vol(A)$ for $\beta \in C'$, where the Zariski closure of $C'$ is $C$.
\end{proof}

%%%%%%%%%%%%%%%%%%%%%%%%%%%%%%%%%%%%%%
\section{A sequence of examples in the simple case}
%%%%%%%%%%%%%%%%%%%%%%%%%%%%%%%%%%%%%%
\label{sec:examples}

In this section, we prove that for any $d\geq 3$, the strict inequality~\eqref{eqn:strict-inequality-for-simple} from Corollary \ref{cor:sharper-upper-bound-F-simple-rank} is sharp for simple parameters $\beta$ as in Definition~\ref{def:simple}, see Theorem \ref{thm:family-examples}.

\begin{theorem}\label{thm:family-examples}
There is a sequence of full rank $(d\times (2d-1))$ integer matrices $\{A_{d,b}\}_{b=2}^\infty$ for which there is a simple parameter $\beta\in\CC^d$ independent of $b$ for which 
\[
\lim_{b\to\infty} 
\frac{\rank (M_{A_{d,b}}(\beta))}{\vol(A_{d,b})} 
= d-1.
\]
In fact, the set of simple parameters $\beta\in \CC^d$ that maximize the ratio $\rank (M_{A_{d,b}}(\beta))/\vol(A_{d,b})$ is a line through the origin.
\end{theorem}

Consider the following $(d\times (2d-1))$-matrix  with $d\geq 3$:
\begin{equation}\label{eqn:Adb}
A_{d,b}
= (a_1 \; a_2\; \cdots a_{2d-1})
\defeq \left(\begin{array}{ccc}
           I_{d-1} & I_{d-1} & \bf{0}_{d-1}\\
           \bf{0}_{d-1}^t & \bf{1}_{d-1}^t & b
          \end{array}\right),
\end{equation}
where $b\geq 2$ is an integer, $I_{d-1}$ denotes the identity matrix of rank $d-1$, $\bf{1}_{d-1}$ is the column vector consisting of $d-1$ entries of $1$, and $\bf{0}_{d-1}$ is the zero column vector of length $d-1$.
          
Note that $\ZZ A_{d,b}=\ZZ^d$. We now compute the normalized volume of $A_{d,b}$ in this lattice.
To do this, for $j\in \ZZ$, set $h^{(j)}\defeq (0,\ldots,0,j)^t$.          
        
\begin{lemma}\label{lem:lemma-volume}
The normalized volume of $A_{d,b}$ in \eqref{eqn:Adb} is $b+d-1$.
\end{lemma}
\begin{proof}
The polytope $\Delta_{A_{d,b}}$ can be decomposed as the union of two polytopes in $\RR^d$ that intersect in a common facet. 
One of these polytopes is the convex hull of the origin, the first $2(d-1)$ columns of $A_{d,b}$, and the lattice point $h^{(1)}$. This is a prism with height $1$ and base equal to a unit $(d-1)$-simplex, so its normalized volume in $\ZZ^d$ is $d$. The second polytope is the convex hull of $h^{(1)}$ and the last $d$ columns of $A_{d,b}$, which is a $d$-simplex. This $d$-simplex is the lattice translation by $h^{(1)}$ of the $d$-simplex that is the convex hull of the origin, the first $(d-1)$-columns of $A_{d,b}$, and $h^{(b-1)}$; therefore, its normalized volume in $\ZZ^d$ is $b-1$.
\end{proof}

\begin{remark}\label{rem:b-copies-of-F}
The last column of $A_{d,b}$ is $b\cdot e_d$, where $e_d$ is the $d$th standard basis vector in $\CC^d$. The face $F_b\defeq \{a_{2d-1}\}\preceq A_{d,b}$ has normalized volume $1$ in the lattice $\ZZ F_b$ and 
\[
\ZZ^d \cap \CC F_b=\bigsqcup_{k=0}^{b-1} \left(h^{(k)}+\ZZ F_b\right)
\] 
consists of $b$ translated copies of $\ZZ F_b$. 
\end{remark}

\begin{remark}\label{rem:remark-minimal-volume}
The normalized volume of $F_b\defeq \{a_{2d-1}\}$ in the lattice $\ZZ^d\cap \QQ F_b$ is $b$. In particular, equality holds in \eqref{eqn:volume-F-ineq} for $A=A_{d,b}$ and $F=F_b$.
\end{remark}

\begin{prop}\label{prop:exceptional}
The exceptional arrangement of $A_{d,b}$ is a finite union of lines parallel to $\CC e_d$.
\end{prop}

\begin{proof}
The first $d-1$ columns of $A_{d,b}$ and its last column are linearly independent, and their nonnegative hull is precisely the first orthant $\RR_{\geq 0}^d$. Thus, $\RR_{\geq 0}A_{d,b}=\RR_{\geq 0}^d$ and $\RR_{\geq 0}A_{d,b}\cap \ZZ^d=\NN^d$.
To determine the set of holes of $\NN A_{d,b}$, given by $\NN^d\setminus \NN A_{d,b}$, which can be written as a finite union of lattice translates of $\NN F_b=\NN b e_d$, 
notice first that the affine semigroup $S\subseteq \NN^d$ generated by the first $2(d-1)$ columns of $A_{d,b}$ is normal, and their lattice span is $\ZZ^d$. Note also that, for $F_b\defeq \{a_{2d-1}\}$, 
\[
\NN A_{d,b}\cap \CC F_b=\NN F_b=\NN b e_d.
\] 
In order to complete the description of $\NN^d\setminus \NN A_{d,b}$, denote by $\Delta_b$ the simplex given by the convex hull of the following points in $\NN A_{d,b}$: 
\[
{\bf 0}, \, 
b a_d=b (e_1+e_d), \, 
b a_{d+1}=b (e_2+e_d),
\, \ldots, \, 
b a_{2d-2}=b(e_{d-1}+e_d), \, 
a_n=b e_d.
\] 
Since $\NN A_{d,b}= S+\NN b e_d$, the set of holes of $\NN A_{d,b}$ is the union of the sets $c+\NN b e_d$, where $c$ runs through the lattice points: 
\[
\ZZ^d\cap \Delta_b\setminus 
\left(
\RR_{\geq 0} (e_1 + e_d )
\,\cup\,
\RR_{\geq 0} (e_2 + e_d )
\,\cup\,  \cdots   \,\cup\,  
\RR_{\geq 0} (e_{d-1}+e_d)
\,\cup\, 
\left(b e_d + \sum_{k=1}^{d-1} \RR_{\geq 0} e_k\right)
\right).
\] 
It now follows from \cite{MMW,berkesch} that the exceptional arrangement of $A_{d,b}$ is 
\[
\cE(A_{d,b})=\bigcup_{k=1}^{d-1} \bigcup_{m=0}^{b-2} (m e_k+\CC F_b).
\qedhere
\]
\end{proof}

By the proof of Proposition~\ref{prop:exceptional}, 
if $b\geq 3$, then all the lattice points in $\Delta_{A_{d,b}}$ belong to $\cE (A_{d,b})$.

\begin{lemma}\label{lem:rank-max}
If $F_b\defeq \{a_{2d-1}\}$, then the function $\beta \in \CC^d \mapsto \rank (M_{A_{d,b}}(\beta))$ reaches its maximum exactly when $\beta \in \CC F_b$, and this maximum value is $(d-1)b+1$.  
\end{lemma}

\begin{proof}
We first show that if $\beta \in \CC F_b$, then the rank of $M_{A_{d,b}}(\beta)$ is $(d-1)b+1$. In this case, the ranking lattices at $\beta$ are 
\[
\EE^{\beta} = 
\bigcup_{j=1}^{b-1} \left(h^{(j)}+\ZZ F_b\right). 
\]
Thus, by \eqref{eqn:formula-simple-rank-jump}, the rank jump at $\beta$ is equal to 
\[
\rank(M_{A_{d,b}}(\beta)) - \vol(A_{d,b}) = 
|B_{F_b}^{\beta}|\cdot \vol_{\ZZ F_b}(F_b)\cdot (\codim (F_b)-1),
\] 
where $\codim (F_b)=d-1$, $\vol_{\ZZ F_b}(F_b)=1$, and $|B_{F_b}^{\beta}|=b-1$ by Remark \ref{rem:b-copies-of-F}, since 
\[
\ZZ F_b\subseteq (\NN A_{d,b} +\ZZ F_b)\cap (\beta+ \CC F)\cap \ZZ^d.
\] 
Thus, 
$\rank (M_{A_{d,b}} (\beta))=\vol_{\ZZ^d}(A_{d,b})+ (b-1)(d-1)$,
which gives the desired equality by Lemma \ref{lem:lemma-volume}. 

In order to prove that this is the maximum value of $\rank (M_{A_{d,b}} (\beta))$, it is enough to observe that when $\beta$ lies in a component of the form $(m e_k+\CC F_b)\subseteq \cE(A_{d,b})$ with $m\neq 0$, then the computation of the rank jump is analogous to the previous case, but the number $|B_{F_b}^{\beta}|$ will be smaller. 
This is the case because 
\[
me_k, \, 
m e_k +e_d, \, 
m e_k +2 e_d, \, \ldots, \, 
m e_k+ m e_d\in \NN A_{d,b},
\] 
and hence there are $(m+1)$ translated copies of $\ZZ F_b$ in $\NN A_{d,b}\cap (m e_k+\CC F_b)$.
\end{proof}

\begin{proof}[Proof of Theorem \ref{thm:family-examples}]
The result now follows immediately from Lemmas \ref{lem:rank-max} and \ref{lem:lemma-volume}.
\end{proof}

%%%%%%%%%%%%%%%%%%%%%%%%%%%%%%%%%%%%%%
%%%%%%%%%%%%%%%%%%%%%%%%%%%%%%%%%%%%%%
\raggedbottom
\def\cprime{$'$} \def\cprime{$'$}
\providecommand{\MR}{\relax\ifhmode\unskip\space\fi MR }
\providecommand{\MRhref}[2]{%
  \href{http://www.ams.org/mathscinet-getitem?mr=#1}{#2}
}
\providecommand{\href}[2]{#2}
%%%%%%%%%%%%%%%%%%%%%%%%%%%%%%%%%%%%%%
%%%%%%%%%%%%%%%%%
%%%%%%%%%%%%%%%%%%%%%%%%%%%%%%%%%%%%%%

%%%%%%%%%%%%%%%%%%%%%%%%%%%%%%%%%%%%%%
\end{document}